\documentclass[a4paper, 12pt,fleqn]{amsart}

\usepackage{amsmath,latexsym}
\usepackage{amssymb}
\usepackage{color}
\usepackage{setspace}
\usepackage[latin1]{inputenc}
\usepackage{rotating}
\usepackage{fancybox}
\usepackage{eurosym}
\usepackage{anysize} % Soporte para el comando \marginsize
\usepackage{graphicx}
\usepackage[all,2cell]{xy}
\usepackage{float}

\usepackage[ruled,vlined]{algorithm2e}

\newtheorem{remark}{Remark}[section]

\newtheorem{theo}{Theorem}[section]

\newtheorem{lemma}{Lemma}[section]

\newcommand{\Q}{\mathbb{Q}}
\newcommand{\Z}{\mathbb{Z}}

\newcommand{\R}{\mathbb{R}}

\newcommand{\dsum}{\displaystyle\sum}
\newcommand{\dlim}{\displaystyle\lim}

\newcommand{\dprod}{\displaystyle\prod}
\newcommand{\dbigcup}{\displaystyle\bigcup}

\marginsize{2.cm}{2.cm}{2.cm}{2.cm} %{izda}{dcha}{sup}{inf}

\title[Short Rational function for MOILP]{Short rational generating functions for multiobjective linear integer programming}

\author{V\'ictor Blanco and Justo Puerto}

\date{Feb-20-2007}

\address{Departamento de Estad\'istica e Investigaci\'on Operativa, Universidad
de Sevilla, 41012 Sevilla, Spain}

\thanks{Facultad de Matem\'aticas.
Universidad de Sevilla. 41012 Seville, SPAIN.}

\email{vblanco@us.es\\
puerto@us.es}

\keywords{Multiple objective optimization, integer programming,
generating functions}

\subjclass[2000]{90C92, 90C10, 05A15}
\begin{document}

\maketitle

\begin{abstract}
This paper presents algorithms for solving multiobjective integer
programming problems. The algorithm uses Barvinok's rational
functions of the polytope that defines the feasible region and
provides as output the entire set of nondominated solutions for
the problem. Theoretical complexity results on the algorithm are
provided in the paper. Specifically, we prove that encoding the
entire set of nondominated solutions of the problem is
polynomially doable, when the dimension of the decision space is
fixed. In addition, we provide polynomial delay algorithms for
enumerating this set. An implementation of the algorithm shows
that it is useful for solving multiobjective integer linear
programs.
\end{abstract}

\section*{Introduction}
Short rational functions were used by Barvinok \cite{barvinok94}
as a tool to develop an algorithm for counting the number of
integer points inside convex polytopes, based in the previous
geometrical papers by Brion \cite{brion88},Khovanskii and Puhlikov
\cite{khovanskii-puhlikov92}, and Lawrence \cite{lawrence91}. The
main idea is encoding those integral points in a rational function
in as many variables as the dimension of the space where the body
lives. Let $P \subset \R^d$ be a given convex polyhedron, the
integral points may be expressed in a formal sum $f(P,z) =
\sum_\alpha z^\alpha$ with $\alpha = (\alpha_1, \ldots, \alpha_d)
\in P\cap\Z^d$, where $z^\alpha = z_1^{\alpha_1}\cdots
z_d^{\alpha_d}$. Barvinok's aimed objective was representing that
formal sum of monomials in the multivariate polynomial ring
$\Z[z_1, \ldots, z_n]$, as a ``short'' sum of rational functions
in the same variables. Actually, Barvinok presented a
polynomial-time algorithm when the dimension, $n$, is fixed, to
compute those functions. A clear example is the polytope $P =
[0,N] \subset \R$: the long expression of the generating function
is $f(P,z) = \sum_{i=0}^N z^i$, and it is easy to see that its
representation as sum of rational functions is the well known
formula $\frac{1-z^{N+1}}{1-z}$.

 Brion proved in 1988 \cite{brion88}, that for computing the short
 generating function of the formal sum associated to a polyhedron, it is enough
 to do it for tangent cones at each vertex of $P$. Barvinok applied this function
  to count the number of integral points inside a polyhedron $P$, that
 is, $\lim_{z\rightarrow (1, \ldots,1)} f(P,z)$, that is not possible
 to compute using the original expression, but it may be
 obtained using tools from complex analysis over the rational function $f$.

The above approach, apart from counting lattice points, has been
used to develop some algorithms to solve, exactly, integer
programming. Actually, De Loera et al \cite{deloera04} and Woods
and Yoshida \cite{woods-yoshida05} presented different methods to
solve this family of problems using Barvinok's rational function
of the polytope defined by the constraints of the given problem.

The goal of this paper is to present new methods for solving
multiobjective integer programming problems. In contrast to usual
integer programming problems, in multiobjective problems there are
at least two (and in this case the problem is called biobjective)
or more objective functions to be optimized.

The importance of multiobjective optimization is not only due to
its theoretical implications but also to its many applications.
Witnesses of that are the large number of real-world decision
problems that appear in the literature formulated as
multiobjective programs. Examples of them are flowshop scheduling
(see ~\cite{ishibuchi-murata98}), analysis in Finance (see
\cite{ehrgott02}, Chapter 20), railway network infrastructure
capacity (see \cite{delorme03}), vehicle routing problems (see
\cite{jozefowiez04, sherbeny01}) or trip organization (see
\cite{steuer85}).

 Multiobjective programs are formulated as
optimization (we restrict ourselves without lost of generality to
the maximization case) problems over feasible regions with at
least two objective functions. Usually, it is not possible to
maximize all the objective functions simultaneously since
objective functions induce a partial order over the vectors in the
feasible region, so a different notion of solution is needed. A
feasible vector is said to be a nondominated (or Pareto optimal)
solution if no other feasible vector has componentwise larger
objective values. The evaluation through the objectives of a
nondominated solution is called efficient solution.

 This paper studies multiobjective integer linear programs (MOILP). Thus, we
assume that there are at least two objective functions involved,
the constraints that define the feasible region are linear, and
the feasible vectors are integers.

Even if we assume that the objective functions are also linear,
there are nowadays relatively few exact methods to solve general
multiobjective integer and linear problems (see \cite{ehrgott02}).
Some of them, as branch and bound with bound sets, which belong to
the class of implicit enumeration methods, combine optimality of
the returned solutions with adaptability to a wide range of
problems (see for example \cite{zionts79, zionts-wallenius80,
marcotte86, mavrotas98} for details). Other methods, as Dynamic
Programming, are general methods for solving, not very
efficiently, general families of optimization problems (see
\cite{karwan-villareal82, daellenbach80}). A different approach,
as the Two-Phase method (see \cite{ulungu-teghem95}), looks for
supported solutions (those that can be found as solutions of a
single-objective problem over the same feasible region but with
objective function a linear combination of the original
objectives) in a first stage and non-supported solutions are found
in a second phase using the supported ones. The Two-phase method
combines usual single-criteria methods with specific
multiobjective techniques.

Apart from those generic methods, there are specific algorithms
for solving some combinatorial biobjective problems: biobjective
knapsacks (\cite{visee98}), biobjective minimum spanning tree
problems (\cite{steiner-radzik08}) or biobjective assignment
problems (\cite{ehrgott08}), as well as heuristics and
metaheuristics algorithms that decrease the CPU time for computing
the nondominated solutions for specific biobjective problems.

Nowadays, new approaches for solving multiobjective problems,
using tools from Algebraic Geometry and Computational Algebra,
have been proposed in the literature aiming to provide new
insights into the combinatorial structure of the problems. This
new research line seems to be prolific in a near future. An
example of that is presented in \cite{blanco-puerto07} where a
notion of partial Gr\"obner basis is given that allows to build a
test family (analogous to the test set concept but for solving
multiobjective problems) to solve general multiobjective linear
integer programming problems.

Another witness of this trend is the recent work by Deloera et al.
\cite{deloera08}. In this paper, the authors present several
algorithms for multiobjective integer linear programs using
generating functions. Nevertheless, their approach differs from
ours in that their requires, in addition, to fix the dimension of
the objective space to prove polynomiality of their algorithms,
and their proofs are totally different. Moreover, no actual
implementation of the algorithms is shown in that paper although
it addresses an interesting shortest distance problems respect to
a prespecified Pareto point.

In this paper, we also use rational generating function of
polytopes for solving multiobjective integer linear programs.

In Section 1, the main results on Barvinok's rational functions,
which we use in our approach, are presented. Section 2 presents
the multiobjective integer problem and the notion of dominance in
order to clarify which kind of solutions we are looking for. The
two following sections analyze different algorithms for solving
general multiobjective problems. In Section 3, fixing the
dimension of the decision space, a polynomial time algorithm that
encodes the set of nondominated solutions of the problem as a
short sum of rational functions is detailed. Next, a digging
algorithm that computes the entire set of nondominated solutions
using the multivariate Laurent expansion for the Barvinok's
function of the polytope defined by the constraints of the problem
is given in Section 4. In that section, a polynomial delay
algorithm for solving multiobjective problems is also presented
(fixing only the dimension of the decision space).

Section 5 shows the results of a computational experiment and its
analysis. Here, we solve biobjective knapsack problems, report on
the performance of the algorithms and draw some conclusions on
their results and their implications.

\section{Barvinok's rational functions}
\label{barvinok}

In this section, we recall some results on short rational
functions for polytopes, that we use in our development. For
details the interested reader is referred to \cite{barvinok94,
barvinok99, barvinok03}.

Let $P = \{x \in \R^n: A\,x \leq b\}$ be a rational polytope in
$\R^n$. The main idea of Barvinok's Theory was encoding the
integer points inside a rational polytope in a ``long" sum of
monomials:
$$
f(P;z) = \dsum_{\alpha\in P\cap\Z^n}\,z^\alpha
$$
where $z^\alpha = z_1^{\alpha_1}\cdots z_n^{\alpha_n}$.

The following results, due to Barvinok, allow us to re-encode, in
polynomial-time for fixed dimension, these integer points in a
``short" sum of rational functions.

\begin{theo}[Theorem 5.4 in \cite{barvinok94}]
\label{theo:barvinok} Assume $n$, the dimension, is fixed. Given a
rational polyhedron $P \subset \R^n$ , the generating function
$f(P;z)$ can be computed in polynomial time in the form
$$
f(P;z) = \dsum_{i\in I} \varepsilon_i
\dfrac{z^{u_i}}{\dprod_{j=1}^n (1-z^{v_{ij}})}
$$
where $I$ is a polynomial-size indexing set, and where
$\varepsilon \in \{1,-1\}$ and $u_i , v_{ij} \in \Z^n$ for all $i$
and $j$.
\end{theo}
As a corollary of this result, Barvinok gave an algorithm for
counting the number of integer points in $P$. It is clear from the
original expression of $f(P;z)$ that this number is
$f(P;\mathbf{1})$, but $\mathbf{1}=(1, \ldots, 1)$ is a pole for
the rational function, so, the number of integer points in the
polyhedron is $\dlim_{z\rightarrow \mathbf{1}} f(S;z)$. This limit
can be computed using residue calculation tools from elementary
complex analysis.

Another useful result due to Barvinok and Wood \cite{barvinok03},
states that computing the short rational function of the
intersection of two polytopes, given the respective short rational
function for each polytope, is doable in polynomial time.

\begin{theo}[Theorem 3.6 in \cite{barvinok03}]\label{theo:intersect}
Let $P_1$, $P_2$ be polytopes in $\R^n$ and $P=P_1 \cap P_2$. Let
$f(P_1; z)$ and $f(P_2; z)$ be their short rational functions with
at most $k$ binomials in each denominator. Then there exists a
polynomial time algorithm that computes
$$
f(P;z) = \dsum_{i\in I} \gamma_i \dfrac{z^{u_i}}{\dprod_{j=1}^s
(1-z^{v_{ij}})}
$$
with $s\leq 2k$, where the $\gamma_i$ are rational numbers and
$u_i$, $v_{ij}$ are nonzero integral vectors for $i \in I$ and
$j=1, \ldots, s$.
\end{theo}

In the proof of the above theorem, the Hadamard product of a pair
of power series is used. Given $g_1(z) =
\dsum_{m\in\Z^d}\beta_{m}\,z^m$ and $g_2(z) =
\dsum_{m\in\Z^d}\gamma_{m}\,z^m$, the Hadamard product $g = g_1
\ast g_2$ is the power series
$$
g(z) = \dsum_{m \in \Z^n} \eta_m\,z^m \qquad \text{where $\eta_m =
\beta_{m}\gamma_{m}$}.
$$

The following Lemma is instrumental to prove Theorem
\ref{theo:intersect}.

\begin{lemma}[Lemma 3.4 in \cite{barvinok03}]
Let us fix $k$. Then there exists a polynomial time algorithm,
which, given functions $g_1(z)$ and $g_2(z)$ such that
\begin{equation}
\label{cond:lemma} g_1(z) = \dfrac{z^{p_1}}{(1-z^{a_{11}})\cdots
(1-z^{a_{1k}})} \quad and \quad g_2(z) =
\dfrac{z^{p_2}}{(1-z^{a_{21}})\cdots (1-z^{a_{2k}})}
\end{equation}
where $p_i, a_{ij} \in \Z^d$ and such that there exists $l \in
\Z^l$ with $\langle l, a_{ij} \rangle <0$ for all $i,j$, computes
a function $h(z)$ in the form
$$
h(z) = \dsum_{i\in I} \beta_i \dfrac{z^{q_i}}{(1-z^{b_{i1}})\cdots
(1-z^{b_is})}
$$
with $q_i, b_{ij} \in \Z^d$, $\beta_i \in \Q$ and $s\leq 2k$ such
that $h$ possesses the Laurent expansion in a neighborhood $U$ of
$z_0 = (e^{l_1}, \ldots, e^{l_n})$ and $h(z) = g_1(z)\ast g_2(z)$.
\end{lemma}

For proving Theorem \ref{theo:intersect}, it is enough to assure
that for given polytopes $P_1, P_2 \subseteq Z^n$, their rational
functions satisfy conditions \eqref{cond:lemma}. It is not
difficult to ensure that the conditions are verified after some
changes are done in the expressions for the short rational
functions (for further details, the interested reader is referred
to \cite{barvinok03}).

Actually, with this result a general theorem can be proved
ensuring that for a pair of polytopes, $P_1, P_2 \subseteq \Z^n$,
there exists a polynomial time algorithm to compute, given the
rational functions for $P_1$ and $P_2$, the short rational
function of any boolean combination of $P_1$ and $P_2$.

Finally, we recall that one can find, in polynomial time, rational
functions for polytopes that are images of polytopes with known
rational function.

\begin{lemma}[Theorem 1.7 in \cite{barvinok94}]
\label{lemma:proj} Let us fix $n$. There exists a number $s=s(n)$
 and a polynomial time algorithm, which, given a rational polytope $P\subseteq \R^n$ and a
linear transformation $T: \R^n \rightarrow \R^r$ such that
$T(\Z^n) \subseteq \Z^r$, computes the function $f(S;z)$ for $S =
T(P \cap \Z^n)$, $S\subseteq \Z^r$ in the form
$$
f(S;z) = \dsum_{i\in I} \alpha_i
\dfrac{z^{p_i}}{(1-z^{a_{i1}})\cdots (1-z^{a_is})}
$$
where $\alpha_i \in \Q$, $p_i, a_{ij} \in \Z^r$ and $a_{ij}\neq 0$
for all $i,j$.
\end{lemma}

To finish this section, we mention the application of short
rational functions to solve single-objective integer programming.
The interested reader is referred to \cite{deloera04,
woods-yoshida05} for further details.

\section{Multiobjective combinatorial optimization problems}
\label{linear}

In this section we present the problem to be solved as well as the
new concept of solutions motivated by the nature of the problem.

A multiobjective integer linear program (MOILP) can be formulated
as:
\begin{align}
 \label{MOILP_gral}
%\begin{split}
\max &\; (c_1\,x, \ldots,  c_k\,x) =: C\,x\nonumber\\
s.t.&\nonumber\\
& \dsum_{j=1}^d \,a_{ij}\,x_j \leq b_i & i=1, \ldots, m\\
& x_j \in \Z_+ & j=1, \ldots, n\nonumber
%\end{split}
\end{align}

with $a_{ij}, b_i$ integers and $x_i$ non negative. Without loss
of generality, we will consider the above problem in its standard
form, i.e., the coefficient of the $k$ objective functions are
non-negative and the constraints are in equation form. In
addition, we will assume that the constraints define a polytope
(bounded) in $\R^n$. Therefore, from now on we deal with
$MOILP_{A,C}(b)$.

It is clear that Problem (\ref{MOILP_gral}) is not a standard
optimization problem since the objective function is a
$k$-coordinate vector, thus inducing a partial order among its
feasible solutions. Hence, solving the above problem requires an
alternative concept of solution, namely the set of nondominated
(or Pareto-optimal) points.

A vector $\widehat{x} \in \R^n$ is said to be a
\textit{nondominated} (or Pareto optimal) solution of
$MIOLP_{A,C}$ if there is no other feasible vector $y$ such that
$$
c_j\,y \geq c_j\,\widehat{x} \qquad \forall j=1, \ldots, k
$$
with at least one strict inequality for some $j$. If $x$ is a
nondominated solution, the vector $Cx = (c_1\,x, \ldots, c_k\,x)
\in \R^k$ is called \textit{efficient}. Note that $X_E$ is a
subset of $\R^n$ (\emph{decision space}) and $Y_E$ is a subset of
$\R^k$ (\emph{objectives space}).

We will say that a dominated point, $y$, is dominated by $x$ if
$c_i\,x \gvertneqq c_i\,y$ for all $i=1,\ldots, k$.\footnote{We
are denoting by $\gvertneqq$ the binary relation ''more than or
equal to" and where it is assumed that at least one of the
inequalities in the list is strict.} We denote by $X_E$ the set of
all nondominated solutions for \eqref{MOILP_gral} and by $Y_E$ the
image under the objective functions of $X_E$, that is, $Y_E =
\{C\,x: x \in X_E\}$.

From the objective function $C$, we obtain a linear partial order
on $\Z^n$ as follows:
$$
x \succ_C y : \Longleftrightarrow C\,x \gvertneqq C\,y
$$
Notice that since $C \in \Z^{k\times n}$, the above relation is
not complete. Hence, there may exist non-comparable vectors. We
will use this partial order, induced by the objective functions of
problem (\ref{MOILP_gral}) as the input for the multiobjective
integer programming algorithm developed in this paper.

Sometimes, the same efficient value is the image of several
nondominated solutions. At this point, different problems can be
tackled. We say that two nondominated solutions, $x^1$ and $x^2$
are equivalent if $C\,x^1 = C\,x^2$. Then, the solutions for
$MOILP_{A,C}(b)$ are one of the following:
\begin{itemize}
\item {\it Complete set}: A subset $X \subseteq X_E$ such that for
all $y \in Y_E$ there is $x\in X$ with $C\,x=y$.
\item {\it Minimal complete set}: A complete set with no
equivalent solutions.
 \item {\it Maximal complete set}: All equivalent solutions.
\end{itemize}
Through this paper, we are looking for the entire set of
nondominated solutions, equivalently the maximal complete set for
$MOILP_{A,C}$.

\section{A short rational function expression of the entire set of nondominated solutions}

We present in this section an algorithm for solving
$MOILP_{A,C}(b)$ using Barvinok's rational functions technique.

\begin{theo}
\label{theo:nondominated-srf} Let $A \in \Z^{m\times n}$, $b \in
\Z^m$, $C =(c_1, \ldots, c_k) \in \Z^{k\times n}$, $J \in \{1,
\ldots, n\}$, and assume that the number of variables $n$ is
fixed. Suppose $P = \{x \in \R^n : A\,x \leq b, x \geq 0\}$ is a
rational convex polytope in $\R^n$. Then, we can encode, in
polynomial time, the entire set of nondominated solutions for
$MOILP_{A,C}(b)$ in a short sum of rational functions.
\end{theo}

\begin{proof}

Using Barvinok's algorithm, compute the following generating
function in $2n$ variables:
\begin{equation}
\label{rf_tf} f(x,y) := \dsum_{(u,v) \in P_{C} \cap \Z^{2n}}
x^u\,y^v
\end{equation}
where $P_{C} = \{(u,v) \in \Z^n \times \Z^n : u, v \in P, c_i\,u -
c_i\,v \geq 0 \text{ for all $i=1, \ldots, k$ and} \dsum_{i=1}^k
c_i\,u - \dsum_{i=1}^k c_i\,v \geq 1\}$. $P_C$ is clearly a
rational polytope. For fixed $u \in \Z^n$, the $y$-degrees,
$\alpha$, in the monomial $x^u\,y^{\alpha}$ of $f(x,y)$ represent
the solutions dominated by $u$.

 Now, for any function $\varphi$, let $\pi_{1, \varphi}, \pi_{2, \varphi}$ be the projections of $\varphi(x,y)$ onto the $x$- and $y$-variables, respectively. Thus $\pi_{2,f}(y)$
 encodes all
dominated feasible integral vectors (because the degree vectors of
the $x$-variables dominate them, by construction), and it can be
computed from $f(x, y)$ in polynomial time by Lemma
\ref{lemma:proj}.

Let $V(P)$ be the set of extreme points of the polytope $P$ and
choose an integer $R \geq \max\{v_i: v\in V(P), i=1,\ldots,n\}$
(we can find such an integer $R$ via linear programming). For this
positive integer, $R$, let $r(x,R)$ be the rational function for
the polytope $\{u \in \R^n_+: u_i\leq R\}$, its expression is:
$$
r(x,R)= \dprod_{i=1}^n \left( \dfrac{1}{1-x_i} +
\dfrac{x_i^R}{1-x_i^{-1}}\right).
$$

Define $f(x,y)$ as above, $\pi_{2,f}(x)$ the projection of $f$
onto the second set of variables as a function of the
$x$-variables and $F(x)$ the short generating function of $P$.
They are computed in polynomial time by Lemma \ref{lemma:proj} and
Theorem \ref{theo:barvinok} respectively. Compute the following
difference:
$$
h(x) := F(x) - \pi_{2,f}(x).
$$

This is the sum over all monomials $x^u$ where $u \in P$ is a
nondominated solution, since we are deleting, from the total sum
of feasible solutions, the set of dominated ones.

This construction gives us a short rational function associated
with the sum over all monomials with degrees being the
nondominated solutions for $MOILP_{A,C}(b)$. As a consequence, we
can compute the number of nondominated solutions for the problem.
The complexity of the entire construction being polynomial since
we only use polynomial time operations among four short rational
functions of polytopes (these operations are the computation of
the short rational expressions for $f(x,y)$, $r(x,R)$ and
$\pi_{2,f}(x)$).
\end{proof}

\begin{remark}
To prove the above result one can use a different approach to
compute the nondominated solutions assuming that there exists a
polynomially bounded (for fixed dimension) feasible lower bound
set, $L$, for $MOILP_{A,C}(b)$, i.e., a set of feasible solutions
such that every nondominated solution is either one element in
$L$, or it dominates at least one the elements in $L$.

First, compute the following operations with generating functions:
$$
H(x,y) = f(x,y) - f(x,y)*(\pi_{2,f}(x)\,r(y,R))
$$

This is the sum over all monomials $x^u\,y^v$ where $u, v \in P$,
$u$ is a nondominated solution and $v$ is dominated by $u$. In
$H(x,y)$, each nondominated solution, $u$, appears as many times
as the number of feasible solutions that it dominates.

Next, compute a feasible lower bound set (see
\cite{fernandez-puerto03, ehrgott-gandibleux07}), $L=\{\alpha_1,
\ldots, \alpha_s\}$. This way the set of nondominated solutions is
encoded using the following construction:

Let $RLB^i(x,y)$ be the following short sum of rational functions
$$
RLB^i(x,y) = H(x,y) \ast (y^{\alpha_i}\,r(x,R)) \qquad i=1,
\ldots, s.
$$

Taking into account that for each $i$, the element $y^{\alpha_i}$
is common factor for $RLB^i(x,y)$ and it is the unique factor
where the $y$-variables appear, we can define $ND^i(x) =
\dfrac{RLB^i(x,y)}{y^{\alpha_i}}, i=1, \ldots, s$, to be the sum
of rational functions that encodes the nondominated solutions that
dominate $\alpha_i$, $i=1, \ldots, s$. Therefore, the entire set
of nondominated solutions for $MOILP_{A,C}(b)$ is encoded in the
short sum of rational functions $ND(x)= \dsum_{i=1}^k\,ND^{i}(x)$.
\end{remark}

\section{Digging algorithm for the set of nondominated solutions of MOILP}
\label{sec:digging}

Section 3 proves that encoding the entire set of efficient
solutions of MOILP can be done in polynomial time for fixed
dimension. This is a compact representation of the solution
concept. Nevertheless, one may be interested in an explicit
description of this list of points. This task could be performed,
by expanding the short rational expression which is ensured by
Theorem \ref{theo:nondominated-srf}, but it would require the
implementation of all operations used in the proof. As far we
know, they have never been efficiently implemented.

An alternative algorithm for enumerating the nondominated
solutions of a multiobjective integer programming problem, which
uses rational generating functions, is the digging algorithm. This
algorithm is an extension of a heuristic proposed by Lasserre
\cite{lasserre03} for the single-objective case.

Let $A, C$ and $b$ be as in Problem (\ref{MOILP_gral}), and assume
that $P=\{ x \in \R^n : Ax \leq b, x\geq 0\}$ is a polytope. Then,
by Theorem \ref{theo:barvinok}, we can compute a rational
expression for $f(P;z) = \sum_{\alpha \in P\cap\Z^n} z^\alpha$ in
the form
$$
f(P;z) = \dsum_{i\in I} \varepsilon_i
\dfrac{z^{u_i}}{\dprod_{j=1}^n (1-z^{v_{ij}})}
$$
in polynomial time for fixed dimension, $n$. Each addend in the
above sum will be referred to as $f_i$, $i\in I$.

If we make the substitution $z_i = z_i\,t_1^{c_{1i}} \cdots
t_k^{c_{ki}}$, in the monomial description we  have $f(P;z, t_1,
\ldots, t_k) = \dsum_{\alpha \in P \cap \Z^n}
z^\alpha\,t_1^{c_1\alpha}\cdots t_k^{c_k\alpha}$, where $c_1,
\ldots, c_k$ are the rows in $C$. It is clear that for enumerating
the entire set of nondominated solutions, it would suffice to look
for the set of leader terms, in the $t$-variables, in the partial
order induced by $C$, $\succ_C$, of the multi-polynomial $f(P;z,
t_1, \ldots, t_k)$. After the above changes we have:
\begin{equation}
\label{eq:changes} f(P;z, t_1, \ldots, t_k) = \dsum_{i\in I}
f_i(P;z, t_1, \ldots, t_k),
\end{equation}
where $ f_i(P;z, t_1, \ldots, t_k):= \varepsilon_i
\dfrac{z^{u_i}\,t_1^{c_1u_i}\cdots t_k^{c_ku_i}}{\dprod_{j=1}^n
(1-z^{v_{ij}}\,t_1^{c_1v_{ij}}\cdots t_k^{c_kv_{ij}})}.$ Now, we
can assume, wlog, that $c_1\,v_{ij}$ is negative or zero. If it
were zero, then we could assume that $c_2\,v_{ij}$ is negative.
Otherwise, we would repeat the argument until the first non zero
element is found (it is assured that this element exists,
otherwise the factor would not appear in the expression of the
short rational function). Indeed, if the first non zero element
were positive, we would make the change:
$$
\dfrac{1}{1-z^{v_{ij}}\,t_1^{c_1v_{ij}}\cdots t_k^{c_kv_{ij}}} =
\dfrac{-z^{-v_{ij}}\,t_1^{-c_1v_{ij}}\cdots
t_k^{-c_kv_{ij}}}{1-z^{-v_{ij}}\,t_1^{-c_1v_{ij}}\cdots
t_k^{-c_kv_{ij}}}
$$
and the sign of the $t_1$-degree would be negative.

With these assumptions, the multivariate Laurent series expansion
for each rational function, $f_i$,  in $f(P;z,t_1, \ldots, t_k)$
is {\tiny $$ \varepsilon_i z^{u_i}\,t_1^{c_1u_i}\cdots
t_k^{c_ku_i} \dprod_{j=1}^d \dsum_{\lambda=0}^{\infty}z^{\lambda
v_{ij}}t_1^{\lambda c_1v_{ij}}\cdots t_k^{\lambda c_kv_{ij}} =
\varepsilon_i z^{u_i}\,t_1^{c_1u_i}\cdots t_k^{c_ku_i}
\dprod_{j=1}^d (1+z^{v_{ij}}t_1^{c_1v_{ij}}\cdots t_k^{c_kv_{ij}}
+ z^{2 v_{ij}}t_1^{2c_1v_{ij}}\cdots t_k^{2c_kv_{ij}} + \cdots )
$$}

The following result allows us to develop a finite algorithm for
solving $MOILP_{A,C}(b)$ using Barvinok's rational generating
functions.

\begin{lemma}\label{lemma:bounds}
Obtaining the entire set of nondominated solutions for a MOILP
requires only an explicit finite, polynomially bounded (in fixed
dimension) number of terms of the long sum in the Laurent
expansion of $f(P;z, t_1, \ldots, t_k)$.
\end{lemma}
\begin{proof}
Let $i\in I$, $j \in \{1, \ldots, n\}$ and define $P_{i}=\{\lambda
\in \Z^n_+ : c_s u_i + \dsum_{r=1}^n \lambda_r\,c_{s}\,v_{ir} \geq
0, s=1, \ldots, k\}$, $M_{ij} = \max\{ \lambda_j : \lambda \in
P_i\}$ and $m_{ij} = \min\{ \lambda_j : \lambda \in P_i\}$.
$M_{ij}$ and $m_{ij}$ are well-defined because $P_i$, defined
above, is non empty and bounded since, by construction, for each
$j \in \{1, \ldots, n\}$ there exists $s \in \{1, \ldots, k\}$
such that $c_s\,v_{ij} < 0$.

Then, it is enough to search for the nondominated solutions in the
finite sum
$$
\varepsilon_i z^{u_i}\,t_1^{c_1u_i}\cdots t_k^{c_ku_i}
\dprod_{j=1}^d \dsum_{\lambda=m_{ij}}^{M_{ij}}t_1^{\lambda
c_1v_{ij}}\cdots t_k^{\lambda c_kv_{ij}}.
$$

Let $U$ (resp. $l$) be the greatest (resp. smallest) value that
appears in the non-zero absolute values of the entries in $A$,
$b$, $C$. Set $M=\max \{U, l^{-1}\}$. First, $m_{ij}\geq 0$. Then,
by applying Cramer's rule one can see that $M_{ij}$ is bounded
above by $O(M^{2n+1})$. Thus, the explicit number of terms in the
expansion of $f_i$, namely $\dprod_{j=1}^n \lfloor M_{ij} -
m_{ij}\rfloor$, is polynomial, when the dimension, $n$ is fixed.
\end{proof}

The digging algorithm looks for the leader terms in the
$t$-variables, with respect to the partial order induced by $C$.
At each rational function (addends in the above sum
(\ref{eq:changes})) multiplications are done in lexicographical
order in their respective bounded hypercubes. If the $t$-degree of
a specific multiplication is not dominated by one of the previous
factors, it is kept in a list; otherwise the algorithm continues
augmenting lexicographically the lambdas. To simplify the search
at each addend, the following consideration can be taken into
account: if $t_1^{\alpha_o+\sum_j\lambda_j\alpha^1_j}\cdots
t_k^{\alpha_o+\sum_j\lambda_j\alpha^k_j}$ is dominated, then any
term of the form $t_1^{\alpha_o+\sum_j\mu_j\alpha^1_j}\cdots
t_k^{\alpha_o+\sum_j\mu_j\alpha^k_j}$, $\mu$ being componentwise
larger than $\lambda$, is dominated as well.

The above process is done on each rational function that appears
in the representation of $f$. As an output we get a set of leader
terms (for each rational function), that are the candidates to be
nondominated solutions. Terms that appear with opposite signs will
be cancelled. Removing terms in the list of candidates (to be
nondominated solutions) implies consideration of those terms that
were dominated by the cancelled ones. These terms are included in
the current list of candidates and the process continues until no
more terms are added.

At the end, some dominated elements may appear in the union of the
final list. Deleting them in a simple cleaning process gives the
list that contains only the entire set of nondominated solutions
for the multiobjective problem.

Algorithm \ref{alg:diggingbiobj} details the pseudocode of the
digging algorithm.

{\sffamily
\begin{algorithm}[!h]
\label{alg:diggingbiobj} \SetLine \SetKwInOut{Input}{input}
\SetKwInOut{Output}{output}

\Input{$A \in \Z^{m\times n}$, $b \in \Z^m$, $C \in \Z^{k\times n}$}

\textbf{Step 1: (Initialization)}

 Compute, $f(z)$, the short sum
of rational functions encoding the set of nondominated solutions
of $MOILP_{A,C}(b)$. The number of rational function is indexed by
$I$.

Make the substitution $z_i = z_i\,t_1^{c_{1i}} \cdots
t_k^{c_{ki}}$ in $f(z)$. Denote by $f_i$, $i \in I$, each one of
the addends in $f$, as in \eqref{eq:changes}.

Set $m_{ij}$ and $M_{ij}$, $j=1, \ldots, n$, the lower and upper
bounds computed in the proof of Lemma \ref{lemma:bounds} and
$\mathcal{S} = \dprod_{j=1}^n [m_{ij}, M_{ij}] \cap \Z^n_+$. Set
$\Gamma_i:=\{\}$, $i \in I$, the initial set of nondominated
solutions encoded in $f_i$.

\textbf{Step 2: (Nondominance test)}

\Repeat{ \Indm{\sffamily No changes in any $\Gamma_i$ are done for
all $i \in I$}}{

\For{$i \in I$}{

\For{ $\lambda^i \in \mathcal{S}$ such that its entries are not
componentwise larger than a previous $\lambda$}{

Compute $p_i:=z^{w_o}\,t_1^{w_1}\cdots t_k^{w_k}$, being $w_o:=u_i
+ \,\dsum_{j=1}^n\,\lambda^i_j\,v_{ij}$ and
$w_h:=c_1\,u_i + \,\dsum_{j=1}^n\,\lambda^i_j\,c_h\,v_{ij} \qquad h=1, \ldots, k$\\

\lIf{$p$ is nondominated by elements in $\Gamma_i$}{$\Gamma_i
\leftarrow \Gamma_i \cup \{p\}$}}}

\textbf{Step 3: (Feasibility test)}

\For{$s, r \in I$, $s<r$}{

\lIf{$p \in \Gamma_j \cap \Gamma_h$, $\varepsilon_j =
-\varepsilon_h$}{$\Gamma_j \leftarrow \Gamma_j \setminus \{p\}$;
$\Gamma_h \leftarrow \Gamma_h\setminus \{p\}$}}

}

Set $\Gamma:=\dbigcup_j \Gamma_j$. Remove from $\Gamma$ the
dominated elements.

\Output{The entire set of nondominated solutions for
$MOILP_{A,C}(b)$: $\Gamma$}

 \caption{Digging algorithm for multiobjective problems}
\end{algorithm}}

Recall that $M=\max \{U, l^{-1}\}$, where $U$ is the greatest
 value that appears in the non-zero absolute values of the entries in $A$, $b$, $C$ and $l$ is the least value
among these values.

Taking into account Lemma \ref{lemma:bounds} and the fact that
Algorithm \ref{alg:diggingbiobj} never cycles, we have the
following statement.

\begin{theo}
\label{theo:M} Algorithm \ref{alg:diggingbiobj} computes in a
finite (bounded on $M$) number of steps, the entire set of
nondominated solutions of the multiobjective Problem
(\ref{MOILP_gral}).
\end{theo}

It is well known that enumerating the nondominated solutions of
MOILP is NP-hard and $\#$P-hard (\cite{ehrgott00,
ehrgott-gandibleux04}). Thus, one cannot expect to have very
efficient algorithms for solving the general problem (when the
dimension is part of the input).

In the following, we concentrate on a different concept of
complexity that has been already used in the literature for
slightly different problems. Computing maximal independent sets on
graphs is known to be $\#$P-hard (\cite{garey-johnson79}),
nevertheless there exist algorithms for obtaining these sets which
ensure that the number of operations necessary to obtain two
consecutive solutions of the problem is bounded by a polynomial in
the problem input size (see e.g. \cite{tsukiyama79}). These
algorithms are called polynomial delay. Formally, an algorithm is
said \emph{polynomial delay} if the delay, which is the maximum
computation time between two consecutive outputs, is bounded by a
polynomial in the input size (\cite{arimura-uno05, johnson88}).

In our case, a polynomial delay algorithm, in fixed dimension, for
solving a multiobjective linear integer program means that once
the first nondominated solution is computed, either in polynomial
time a next nondominated solution is found or the termination of
the algorithm is given as an output.

Next, we present a polynomial delay algorithm, in fixed dimension,
for solving multiobjective integer linear programming problems.
This algorithm combines the theoretical construction of Theorem
\ref{theo:nondominated-srf} and a digging process in the Laurent
expansion of the short rational functions of the polytope
associated with the constraints of the problem.\medskip\\

\hspace*{1cm}The algorithm proceeds as follows.

Let $f(z)$ be the short rational function that encodes the
nondominated solutions (by Theorem \ref{theo:nondominated-srf},
the complexity of computing $f$ is polynomial -in fixed
dimension-). Make the changes $z_i = z_i\,t_1^{c_{1i}} \cdots
t_k^{c_{ki}}$, for $i \in I$, in $f$. Denote by $f_i$ each of the
rational functions of $f$ after the above changes. Next, the
Laurent expansion over each rational function, $f_i$, is done in
the following way: (1) Check if $f_i$ contains nondominated
solutions computing the Hadamard product of $f_i$ with $f$. If
$f_i$ does not contain nondominated solutions, discard it and set
$I:=I\backslash\{i\}$ (termination); (2) if $f_i$ encodes
nondominated solutions, look for an arbitrary nondominated
solution (expanding $f_i$); (3) once the first nondominated
solution, $\alpha$, is found, check if there exist more
nondominated solutions encoded in the same rational function
computing $f \ast (f_i - z^\alpha\,t_1^{c_1\alpha}\cdots
t_k^{c_k\alpha})$. If there are more solutions encoded in $f_i$,
look for them in $f_i - z^\alpha\,t_1^{c_1\alpha}\cdots
t_k^{c_k\alpha}$. Repeat this process until no new nondominated
solutions can be found in $f_i$.

The process above describes the pseudocode written in Algorithm
\ref{alg:poldelay}. {\sffamily
\begin{algorithm}[h]
\label{alg:poldelay} \SetLine \SetKwInOut{Input}{input}
\SetKwInOut{Output}{output}

\Input{$A \in \Z^{m\times n}$, $b \in \Z^m$, $C \in \Z^{k\times n}$}
\Output{The entire set of nondominated ($X_E$) and efficient
($Y_E$) solutions for $MOILP_{A,C}(b)$}

Set $X_E = \{\}$ and $Y_E=\{\}$.

\textbf{Step 1: } Compute, $f(z)$, the short sum of rational
functions encoding the set of nondominated solutions of
$MOILP_{A,C}(b)$. The number of rational functions is indexed by
$I$.

Make the substitution $z_i = z_i\,t_1^{c_{1i}} \cdots
t_k^{c_{ki}}$ in $f(z)$. Denote by $f_i$, $i \in I$, each one of
the addends in $f$ ($f = \dsum_{i\in I} f_i$).

\textbf{Step 2: } For each $i \in I$, check $f_i\ast f$. If the
set of lattice points encoded by this rational function is empty,
do $I \leftarrow I\setminus \{i\}$.

\textbf{Step 3: }\\
\noindent\While{$I \neq \emptyset$}{\For{$i \in I$}{ Look for the
first nondominated solution,
$\alpha$, that appears in the Laurent expansion of $f_i$.\\
Set $X_E \leftarrow X_E \cup \{\alpha\}$ and $Y_E \leftarrow Y_E
\cup \{C\,\alpha\}$.

Set $f_i \leftarrow f_i - z^\alpha\,t_1^{c_1\alpha}\cdots
t_k^{c_k\alpha}$\\
and check if $f \ast f_i$ encodes lattice points. If it does not
encode lattice points, discard $f_i$ ($I \leftarrow I\setminus
\{i\}$) since $f_i$ does not encode any other nondominated point,
otherwise repeat. }}

\caption{A polynomial delay algorithm for solving MOILP}
\end{algorithm}}

\begin{theo}
Assume $n$ is a constant. Algorithm \ref{alg:poldelay} provides a
polynomial delay procedure to obtain the entire set of
nondominated solutions of $MOILP_{A,C}(b)$.
\end{theo}

\begin{proof}
Let $f$ be the rational function that encodes the nondominated
solutions of $MOILP_{A,C}(b)$. Theorem \ref{theo:nondominated-srf}
ensured that $f$ is a sum of short rational functions that can be
computed in polynomial time.

Algorithm \ref{alg:poldelay} digs separately on each one of the
rational functions $f_i$, $i \in I$,that define $f$. (Recall that
$f= \dsum_{i\in I} f_i$).

Fix $i\in I$. First, the algorithm checks whether $f_i$ encodes
some nondominated solutions. This test is doable in polynomial
time by Theorem \ref{theo:intersect}. If the answer is positive,
an arbitrary nondominated solution is found among those encoded in
$f_i$. This is done using digging and the Intersection Lemma.
Specifically, the algorithm expands $f_i$ on the hyperbox
$\dprod_{j=1}^n [m_{ij}, M_{ij}] \cap \Z^n$ and checks whether
each term is nondominated. The expansion is polynomial, for fixed
$n$, since the number of terms is polynomially bounded by Lemma
\ref{lemma:bounds}. The test is performed using the Hadamard
product of each term with $f$.

The process is clearly a polynomial delay algorithm. We use
digging separately on each rational function $f_i$ that encodes
nondominated points. Thus, the time necessary to find a new
nondominated solution from the last one is bounded by the
application of digging on a particular $f_i$ which, as argued
above, is polynomially bounded.
\end{proof}

Instead of the above algorithm one can use a binary search
procedure to solve multiobjective problem using short generating
functions. In the worst case, digging algorithm may need to expand
every nonnegative term to obtain the set of nondominated
solutions. Therefore, as it is stated in Theorem \ref{theo:M}, the
number of steps to solve the problem can be polynomially bounded
on $M$. With a binary search approach, the number of steps to
obtain consecutive solutions of our problem decreases to a number
polynomially bounded on $log(M)$. A binary search approach was
already used in \cite{deloera08}. Here, the novelty is that our
analysis does not require to fix the dimension of the objective
space whereas in \cite{deloera08} it was required.

The process is as follows. Let $M$ be defined as above. By
construction $P\subseteq [0,M]^n$. We proceed by dividing the
hypercube $[0,M]^n$ into $2^n$ hypercubes of smaller dimensions,
and recursively repeating the division process over those
hypercubes containing at least one nondominated solution (until
only one solution is included in each element of the partition),
whereas those hypercubes that at a given stage of the process do
not contain nondominated solutions are discarded for any further
consideration.

The division process is done by bisecting each dimension. Testing
for nondominated solutions on a given hypercube (at any stage of
the process) is always done using the same tool based on Theorem
\ref{theo:nondominated-srf}. That result allows us to construct,
in polynomial time in fixed dimension, the function $h(x)$ that
encodes all nondominated solutions. Moreover, it is easy to see
that the short rational function that encodes the integer points
in the hypercube $\mathcal{H} = \dprod_{i=1}^n [m_i, M_i]$, with
$m_i, M_i \in \Q$, $i=1, \ldots, n$, is:
$$
r_\mathcal{H}(x)=\dprod_{i=1}^n\big[\dfrac{x_i^{m_i}}{1-x_i} +
\dfrac{x_i^{M_i}}{1-x_i^{-1}}\big]
$$
Thus, the Hadamard product, $h(x)\ast r_\mathcal{H}(x)$ encodes
the subset of nondominated solutions that lie in $\mathcal{H}$;
and hence by Barvinok's theory we can also count, in polynomial
time, the number of integer points encoded by $h(x)\ast
r_\mathcal{H}(x)$ (Lemma 3.4 in \cite{barvinok03}).

The elements in our search space (hypercubes) are organized on a
search tree and we use a depth first search strategy. Each node is
a hypercube containing nondominated solutions. Descendants of a
given node are hypercubes obtained bisecting the edges on the
previous one (parent). It is clear that the maximum depth of the
tree is $O(logM)$. The above construction ensures that, provided
that the set of nondominated solutions is nonempty, finding a
first nondominated solution can be done testing at most $O(2^n
logM)$ nodes in the search tree. Since testing a node is
polynomial, in fixed dimension, this operation is polynomial.
Moreover, finding a new nondominated solution from a given one is
also polynomial. Indeed, it consists of backtracking at most
$O(logM)$ nodes until we find a branch containing nondominated
points and then we have to explore, at most, $O(2^n logM)$ nodes;
or detecting that none of the branches contain solutions.

An illustrative example of this procedure is shown in Figure
\ref{tree} where can be seen how the initial hypercube,
$[0,4]\times [0,4]$, is divided successively in  sub-hypercubes,
until an isolated nondominated solution is located in one of them.

\begin{figure}[h]
\includegraphics[scale=0.4]{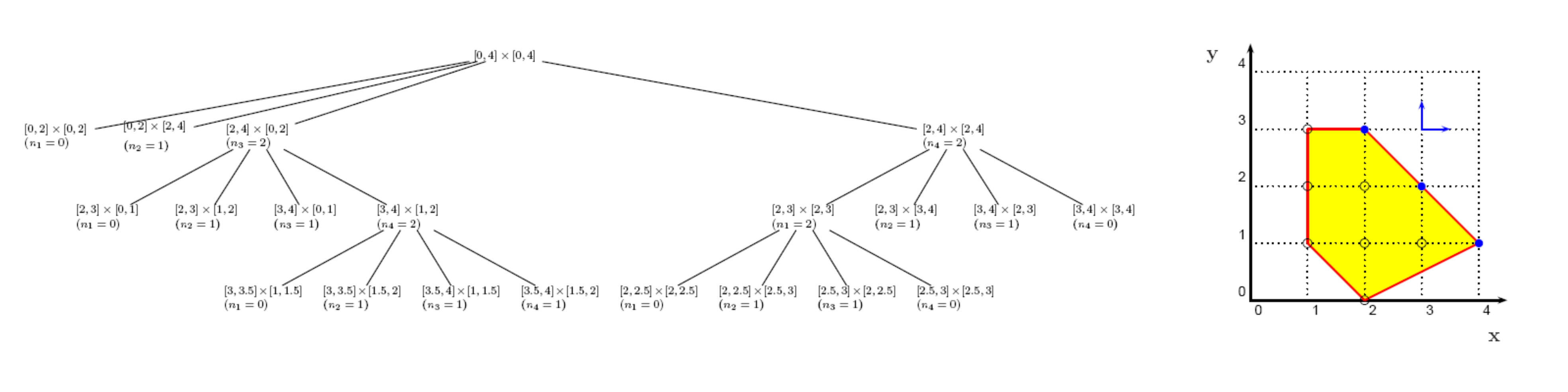}
\caption{\label{tree} Search tree for the problem $v-\max\{(x,y):
x+y \leq 5, x-2y\leq 2, x+y \geq 2, x \geq 1, y \leq 3, x, y \in
\Z_+\}$}
\end{figure}

The finiteness of this procedure is assured since the number of
times that the hypercube $[0,M]^n$ can be divided in $2^n$
sub-hypercubes is bounded by $log(M)$.

The pseudocode for this procedure is shown in Algorithm
\ref{alg:logM}.

{\sffamily
\begin{algorithm}[!h]
\label{alg:logM} \SetLine \SetKwInOut{Input}{input}
\SetKwInOut{Output}{output} \textbf{Initialization: }
$\mathcal{M}=[0,M]^n \subseteq P$.

\textbf{Step 1: } Let $\mathcal{M}_1, \ldots, \mathcal{M}_{2^n}$
be the hypercubes obtained dividing $\mathcal{M}$ by its central
point.

$i=1$

 \textbf{Step 2: }
\Repeat{$i \leq 2^n$}{Count $n_{\mathcal{M}_i}$, the number of
integer points encoded in $r_{\mathcal{M}_i}(x)\ast h(x)$. This is
the number of nondominated solutions in the hypercube
$\mathcal{M}_i$.

\uIf{$n_{\mathcal{M}_i}=0$}{

\lIf{$i<2^n$}{$i\leftarrow i+1$}

\lElse{Go to Step 1 with $\mathcal{M}$ the next element in the
search tree, using depth first search.}}
\uElseIf{$n_{\mathcal{M}_i}=1$ ( and $P \cap \mathcal{M}_i = \{
x^*\}$)}{$ND = ND \cup \{x^*\}$ and $i\leftarrow i+1$.} \Else{ Go
to Step 1 with $\mathcal{M}=\mathcal{M}_i$.}} \caption{}
\end{algorithm}}

\begin{theo}
Assume $n$ is a constant. Algorithm \ref{alg:logM} provides a
polynomial delay (polynomially bounded on $log(M)$) procedure to
obtain the entire set of nondominated solutions of
$MOILP_{A,C}(b)$.
\end{theo}

\begin{remark}
\label{remark1} The application of the above algorithm to the
single criterion case provides an alternative proof of
polynomiality for the problem of finding an optimal solution of
integer linear problems, in fixed dimension.

Assume that the number of objectives, $k$, is $1$, and that there
exists a unique optimal value for the problem. Applying Theorem
\ref{theo:nondominated-srf} ensures that the optimal solution  of
the problem is found in polynomial time, if the dimension $n$ is
fixed.
\end{remark}

\begin{remark}[Optimization over the set of nondominated
solutions] In practice, a decision maker expects to be helped by
the solutions of the multiobjective problem. In many cases, the
set of nondominated solutions is too large to make easily the
decision, so it is necessary to optimize (using a new criterion)
over the set of nondominated solutions.

With our approach, we are able to compute, in polynomial time for
fixed dimension, a ``short sum of rational
functions"-representation, $F(z)$, of the set of nondominated
solutions of $MOILP_{A,C}(b)$. This representation allows us to
re-optimize with a linear objective, $\nu$, based in the
algorithms for solving single-objective integer programming
problems using Barvinok's functions (see e.g.
\cite{woods-yoshida05}) or the algorithm proposed in Remark
\ref{remark1}. The above discussion proves that solving the
problem of optimizing a linear function  over the efficient region
of a multiobjective problem $MOILP_{A,C}(b)$ is doable in
polynomial time, for fixed dimension.
\end{remark}

\section{Computational Experiments}
For illustrative propposes, a series of computational experiments
have been performed in order to evaluate the behavior of a simple
implementation of the digging algorithm (Algorithm
\ref{alg:diggingbiobj}). Computations of short rational functions
have been done with Latte v1.2 \cite{deloera03a} and Algorithm
\ref{alg:diggingbiobj} has been coded in MAPLE 10 and executed in
a PC with an Intel Pentium 4 processor at 2.66Gz and 1 GB of RAM.
The implementation has been done in a symbolic programming
language, available upon request, in order to make the access easy
for the interested readers.

The performance of the algorithm was tested on randomly generated
instances for biobjective (two objectives) knapsack problems.
Problems from 4 to 8 variables were considered, and for each
group, the coefficients of the constraint were randomly generated
in $[0,20]$. The coefficients of the two objective matrices range
in $[0,20]$ and the coefficients of the right hand side were
randomized in $[20,50]$. Thus, the problems solved are in the
form:
\begin{equation}
 \max\, (c_1, c_2) \,x \quad s.t.\quad a_1x_1 + \cdots + a_n\,x_n \leq
b, x_i \in \Z_+
\end{equation}

The computational tests have been done on this way for each number
of variables: (1) Generate 5 constraint vectors and right hand
sides and compute the shorts rational functions for each of them;
(2) Generate a random biobjective matrix and run digging algorithm
for them to obtain the set of nondominated solutions.

Table \ref{comp:knapsack} contains a summary of the average
results obtained for the considered knapsack multiobjective
problems. The second and third columns show the average CPU times
for each stage in the Algorithm: \verb"srf" is the CPU time for
computing the short rational function expression for the polytope
with LattE and  \verb"mo-digging" the CPU time for running the
multiobjective digging algorithm for the problem. The total
average CPU times are summarized in the \texttt{total} column.
Columns \texttt{latpoints} and \texttt{nosrf} represent the number
of lattice points in the polytope and the number of short rational
functions, respectively. The average number of efficient solutions
that appear for the problem is presented under \texttt{effic}. The
problems have been named as \verb"knapN" where \verb"N" is the
number of variables of the biobjective knapsack problem.

\begin{table}[h]{\small \center{
\begin{tabular}{lrrrrrr}\hline
\texttt{problem} & \texttt{srf} & \texttt{latpoints} &
\texttt{nosrf} & \texttt{mo-digging} & \texttt{effic} &
\texttt{total}\\\hline
\verb"knap4" & 0.018 & 12.25 & 25.75 & 4.863 & 4.5 & 4.881 \\
\verb"knap5" & 0.038 & 31 & 62.5 & 487.640 & 9.25 & 487.678\\
\verb"knap6" & 0.098 & 217.666 & 124.25 & 2364.391 & 7.666 & 2364.489\\
\verb"knap7" & 0.216 & 325 & 203 & 2869.268 & 20 &
2869.484\\
\verb"knap8" & 0.412 & 3478 & 342 & 10245.533 & 46 &
10245.933\\\hline
\end{tabular}
\caption{\label{comp:knapsack}Summary of computational experiments
for knapsack problems}}}
\end{table}

As can be seen in Table \ref{comp:knapsack}, the computation times
are clearly divided into two steps (\texttt{srf} and
\texttt{mo-digging}), being the most expensive the application of
the digging algorithm (Algorithm \ref{alg:diggingbiobj}). In all
cases more than 99\% of the total time is spent expanding the
short rational function using ``digging algorithm''.

The CPU times and sizes in the two steps are highly sensitive to
the number of variables. It is clear that one cannot expect fast
algorithm for solving MOILP, since all these problems are NP-hard
and \#P-hard. Nevertheless, this approach gives exact tools for
solving any MOILP problem, independently of the combinatorial
nature of the problem.

Finally, from our computational experiments, we have detected that
an easy, promising heuristic algorithm could be obtained
truncating the expansion at each rational function. That algorithm
would accelerate the computational times at the price of obtaining
only heuristics nondominated points.

\section*{Acknowledgements} This research has been partially supported by Junta de Andalucía grant number {\scriptsize P06--FQM--01366} and by the Spanish Ministry of Science and Education grant
number {\scriptsize MTM2007--67433--C02--01}. The authors
acknowledge the useful comments received from J. De Loera and M.
Köppe on an earlier version of this paper.


\begin{thebibliography}{99}


\bibitem{arimura-uno05} Arimura, H., Uno, T. (2005). A Polynomial Space and Polynomial Delay Algorithm for Enumeration of Maximal Motifs in a Sequence. ISAAC 2005: 724--737


\bibitem{barvinok94}  Barvinok, A. A polynomial time algorithm for counting integral points in polyhedra when the dimension is fixed,  Mathematics of Operations Research ,  19  (1994), 769--779.

\bibitem{barvinok99}  Barvinok, A. and Pommersheim, J.E.  An algorithmic theory of lattice points in polyhedra, in: New Perspectives in Algebraic Combinatorics (Berkeley, CA, 1996-1997), 91-147, Math. Sci. Res. Inst. Publ. 38, Cambridge Univ. Press, Cambridge, 1999.

\bibitem{barvinok03}  Barvinok, A. and Woods, K. Short rational generating functions for lattice point problems, Journal of the American Mathematical Society,  16  (2003), 957--979.

\bibitem{blanco-puerto07} Blanco, V. and Puerto, J. (2007). Partial Gr\"obner bases for
multiobjective combinatorial optimization. Submitted.
arXiv:0709.1660

\bibitem{brion88} Brion, M. Points entiers dans les poly\`edres convexes. Annales scientifiques de l'\`Ecole Normale Sup\`erieure S\'er. 4, 21 no. 4 (1988), p. 653--663.


\bibitem{daellenbach80} Daellenbach, H.G., C.A. De Kluyver (1980) . Note on multiple
objective dynamic programming. Journal of the Operational Research
Society 31 591--594.

\bibitem{deloera03a} De Loera, J.A., Haws, D., Hemmecke, R., Huggins, P., Tauzer,
J., Yoshida, R. A User's Guide for LattE v1.1. 2003, software
package LattE is available at http://www.math.ucdavis.edu/latte/

\bibitem{deloera04} De Loera, J.A, Haws, D., Hemmecke, R., Huggins, P., Sturmfels, B.,
and Yoshida, R. (2004). Short rational functions for toric algebra
and applications. Journal of Symbolic Computation, Vol. 38, 2 ,
2004, 959--973.

\bibitem{deloera08} De Loera, J.A., Hemmecke, R., Köppe, M. (2008). Pareto Optima of Multicriteria Integer Linear Programs. To appear in: INFORMS Journal on Computing, 2008

\bibitem{delorme03}Delorme, X. , Gandibleux, X. and Degoutin, F. (2003).
Resolution approch\'e du probleme de set packing bi-objectifs. In
Proceedings de l'ecole d'Automne de Recherche Operationnelle de
Tours (EARO), 74--80.

\bibitem{ehrgott00} Ehrgott, M. (2000). Approximation algorithms for combinatorial multicriteria optimization problems.
International Transactions in Operational Research 7:5--31.

\bibitem{ehrgott02} Ehrgott, M. and Gandibleux, X. (editors) (2002). Multiple Criteria
Optimization. State of the Art Annotated Bibliographic Surveys.
Boston, Kluwer.

\bibitem{ehrgott-gandibleux04} Ehrgott, M. and Gandibleux, X. (2004). Approximative solution methods for
multiobjective combinatorial optimization. TOP 12(1):1--88.

\bibitem{ehrgott-gandibleux07} Ehrgott, M. and Gandibleux, X. (2007). Bound sets for biobjective
combinatorial optimization problems. Comput. Oper. Res. 34, 9
(Sep. 2007), 2674-2694.

\bibitem{fernandez-puerto03} Fernández, E. and Puerto, J. (2003). The multiobjective solution of the uncapacitated plant location problem. European Journal of Operational Research. vol. 45, n.3 509-529.

\bibitem{garey-johnson79} Garey, M. R. and Johnson, D. S. (1979). Computers and Intractability: a Guide to the Theory of Np-.Completeness. W. H. Freeman \& Co.

\bibitem{ishibuchi-murata98} Ishibuchi, H. and Murata, T. (1998),  A multi-objective genetic local search algorithm and its application to flowshop scheduling, IEEE Trans. Syst., Man, Cybern. C 28, 392--403.


\bibitem{johnson88} Johnson, D. S. and Papadimitriou, C. H. (1988). On generating all maximal independent sets. Inf. Process. Lett. 27, 3 (Mar. 1988), 119--123.

\bibitem{jozefowiez04} Jozefowiez, N. , Semet, F. and Talbi, E-G. (2004). A multi-objective evolutionary algorithm for the covering tour problem, Chapter 11 in "Applications
 of multi-objective evolutionary algorithms", C. A. Coello and G. B. Lamont (editors), p 247-267, World Scientific.

\bibitem{karwan-villareal82} Villarreal, B. and Karwan, M.H. (1982).
Multicriteria Dynamic Programming with an Application to the
Integer Case. Journal of Optimization Theory and Applications.
Vol. 31. pp 43-69.

\bibitem{khovanskii-puhlikov92} Khovanskii, A.G.  and Pukhlikov, A.V. , The Riemann-Roch theorem
for integrals and sums of quasipolynomials on virtual polytopes,
(Russian) Algebra i Analiz 4 (1992), no. 4, 188--216; translation
in St. Petersburg Mathematical Journal, 4 (1993), no. 4, 789--812.

\bibitem{lasserre03} Lasserre, J.B. Integer programming, Barvinok's counting algorithm and Gomory relaxations. Operations Research Letters, 32, 2003, 133--137.

\bibitem{lawrence91} Lawrence,J. , Rational-function-valued valuations on polyhedra, in:
Discrete and Computational Geometry (New Brunswick, NJ,
1989/1990), 199--208, DIMACS Ser. Discrete Mathematics and
Theoretical Computer Science, 6,
 American Mathematical Society, Providence, RI, 1991.

\bibitem{lenstra81} Lenstra, H.W. (Jr.) (1981). Integer programming with a fixed number of variables, Report 81--03, Mathematisch Instituut, Universiteit ban
Amsterdam.

\bibitem{marcotte86} Marcotte, O., R.M. Soland (1986). An interactive
branch-and-bound algorithm for multiple criteria optimization.
Management Science 32 61--75.

\bibitem{mavrotas98} Mavrotas, G., D. Diakoulaki (1998). A branch and bound algorithm for mixed zero-one
multiple objective linear programming. European Journal of
Operational Research 107 530--541.

\bibitem{ehrgott08} Przybylski, A., Gandibleux, X. and Ehrgott, M. (2008). Two phase
algorithms for the bi-objective assignment problem. European
Journal of Operational Research 185(2):509-533, 2008.

\bibitem{sherbeny01} El-Sherbeny, N. (2001). Resolution of a Vehicle Routing
Problem with Multiobjective Simulated Annealing Method, PhD
thesis, Faculte Polytechnique de Mons, Belgium.

\bibitem{steiner-radzik08} Steiner, S. and Radzik, T. 2008. Computing all efficient
solutions of the biobjective minimum spanning tree problem.
Comput. Oper. Res. 35, 1 (Jan. 2008), 198-211.

\bibitem{steuer85} Steuer, R.E. (1985). Multiple Criteria Optimization: Theory,
Computation and Application. John Wiley \& Sons, New York, NY.

\bibitem{tsukiyama79}  Tsukiyama, S., Ide, M., Ariyoshi, H. and Shirakawa, I. (1979). A new algorithm for generating all maximal independent sets. SIAM J. Comput. 6, pp.
505--517.

\bibitem{ulungu-teghem95} Ulungu, E. and Teghem, J. (1995). The two-phases method: An
efficient procedure to solve biobjective combinatorial
optimization problems, Foundations of Computing and Decision
Sciences 20(2), 149--165.

\bibitem{visee98} Visée, M., Teghem, J., Pirlot, M., and Ulungu, E. L. (1998).
Two-phases Method and Branch and Bound Procedures to Solve the
Biobjective Knapsack Problem. J. of Global Optimization 12, 2,
139--155.

\bibitem{woods-yoshida05} Woods, K. and Yoshida, R. (2005). Short rational generating functions
and their applications to integer programming , SIAG/OPT Views and
News, 16 , 15-19.

\bibitem{zionts79} Zionts, S. (1979). A survey of multiple criteria integer programming methods.
Annals of Discrete Mathematics 5, 389--398.

\bibitem{zionts-wallenius80} Zionts, S. and Wallenius, J. (1980). Identifying efficient vectors:
some theory and computational results. Operations Research 23,
785--793.

\end{thebibliography}
\end{document}